\documentclass[12pt, reqno]{amsart}
\usepackage{amsmath, amsthm, amscd, amsfonts, amssymb, graphicx, color}
\usepackage[bookmarksnumbered, colorlinks, plainpages]{hyperref}

\newtheorem{theorem}{Theorem}[section]

\newtheorem{proposition}[theorem]{Proposition}
\newtheorem{corollary}[theorem]{Corollary}
\theoremstyle{definition}
\newtheorem{definition}[theorem]{Definition}
\newtheorem{example}[theorem]{Example}

\theoremstyle{remark}
\newtheorem{remark}[theorem]{Remark}
\numberwithin{equation}{section}

\begin{document}

\title{$(\sigma,\tau)$-amenability of $C^*$-algebras}
\author[M.\ Mirzavaziri, M.S.\ Moslehian]{M.\ Mirzavaziri and M.\ S.\ Moslehian}
\address{Department of Pure Mathematics, Ferdowsi University of Mashhad, P. O. Box 1159, Mashhad 91775, Iran;
\newline Centre of Excellence in Analysis on Algebraic Structures (CEAAS), Ferdowsi University of Mashhad, Iran.}
\email{madjid@mirzavaziri.com and mirzavaziri@math.um.ac.ir}
\email{moslehian@ferdowsi.um.ac.ir and moslehian@ams.org}

\begin{abstract}
Suppose that ${\mathcal A}$ is an algebra, $\sigma,\tau:{\mathcal
A}\to{\mathcal A}$ are two linear mappings such that both
$\sigma({\mathcal A})$ and $\tau({\mathcal A})$ are subalgebras of
${\mathcal A}$ and ${\mathcal X}$ is a $\big(\tau({\mathcal
A}),\sigma({\mathcal A})\big)$-bimodule. A linear mapping
$D:{\mathcal A}\to {\mathcal X}$ is called a
$(\sigma,\tau)$-derivation if $D(ab)=D(a)\cdot\sigma(b)+\tau(a)\cdot
D(b)\\\,(a,b\in {\mathcal A})$. A $(\sigma,\tau)$-derivation $D$ is
called a $(\sigma,\tau)$-inner derivation if there exists an
$x\in{\mathcal X}$ such that $D$ is of the form either
$D_x^-(a)=x\cdot \sigma(a)-\tau(a)\cdot x\,\,(a\in {\mathcal A})$ or
$D_x^ +(a)=x\cdot \sigma(a)+\tau(a)\cdot x \,\,(a\in {\mathcal A})$.
A Banach algebra ${\mathcal A}$ is called $(\sigma,\tau)$-amenable
if every $(\sigma,\tau)$-derivation from ${\mathcal A}$ into a dual
Banach $\big(\tau({\mathcal
A}),\sigma({\mathcal A})\big)$-bimodule is $(\sigma,\tau)$-inner.\\
Studying some general algebraic aspects of
$(\sigma,\tau)$-derivations, we investigate the relation between
amenability and $(\sigma,\tau)$-amenability of Banach algebras in
the case when $\sigma, \tau$ are homomorphisms. We prove that if
$\mathfrak A$ is a $C^*$-algebra and $\sigma, \tau$ are
$*$-homomorphisms with $\ker(\sigma)=\ker(\tau)$, then ${\mathfrak
A}$ is $(\sigma, \tau)$-amenable if and only if $\sigma({\mathfrak
A})$ is amenable.
\end{abstract}

\maketitle

\section{Introduction and preliminaries}

The notion of an amenable Banach algebra was introduced by B.E.~
Johnson in his definitive monograph \cite{J1}. This class of Banach
algebras arises naturally out of the cohomology theory for Banach
algebras, the algebraic version of which was developed by
G.~Hochschild \cite{H1}. For a comprehensive account on amenability
the reader is referred to books \cite{DAL, RUN}.

\noindent Suppose that ${\mathcal A}$ is an algebra,
$\sigma,\tau:{\mathcal A}\to{\mathcal A}$ are two linear mappings
such that both $\sigma({\mathcal A})$ and $\tau({\mathcal A})$ are
subalgebras of ${\mathcal A}$ and ${\mathcal X}$ is a
$\big(\tau({\mathcal A}),\sigma({\mathcal A})\big)$-bimodule. A
linear mapping $D:{\mathcal A}\to {\mathcal X}$ is called a
$(\sigma,\tau)$-derivation if
\[D(ab)=D(a)\cdot\sigma(b)+\tau(a)\cdot D(b)\qquad (a,b\in {\mathcal A}).\]
We say that a $(\sigma,\tau)$-derivation $D$ is a
$(\sigma,\tau)$-inner derivation if it is of the form either
$D_x^-(a)=x\cdot\sigma(a)-\tau(a)\cdot x\,\,(a\in {\mathcal A})$ or
$D_x^+(a)=x\cdot\sigma(a)+\tau(a)\cdot x\,\,(a\in {\mathcal A})$ for
some $x\in{\mathcal X}$. For other approaches to generalized
derivations and their applications see \cite{B-M, BRE, B-V, M-M4}
and references therein. In particular, the automatic continuity
problem for $(\sigma,\tau)$-derivations is considered in \cite{M-M1,
M-M2, M-M3} and an achievement of continuity of
$(\sigma,\tau)$-derivations without linearity is given in
\cite{H-J-M-M}.

\noindent A wide range of examples are as follows (see \cite{M-M2}):
\begin{itemize}
\item[(i)]{Every ordinary derivation is an
$id_{{\mathcal A}}$-derivation, where $id_{{\mathcal A}}$ is the
identity map on the algebra ${\mathcal A}$.}
\item[(ii)]{Every endomorphism $\alpha$ on ${\mathcal A}$ is an
$\frac{\alpha}{2}$-derivation on ${\mathcal A}$.}
\item[(iii)]{A $\theta$-derivation is nothing than a
point derivation $d:{\mathcal A}\to{\mathbb C}$ at the character
$\theta$.}
\end{itemize}

\noindent A Banach algebra ${\mathcal A}$ is said to be
\textit{$(\sigma,\tau)$-amenable} (resp.
\textit{$(\sigma,\tau)$-contractible}) if every continuous
$(\sigma,\tau)$-derivation from ${\mathcal A}$ into a dual Banach
$\big(\tau({\mathcal A}),\sigma({\mathcal A})\big)$-bimodule
${\mathcal X}^*$ (a Banach $\big(\tau({\mathcal A}),\sigma({\mathcal
A})\big)$-bimodule ${\mathcal X}$, resp.) is $(\sigma,\tau)$-inner.
Recall that the bimodule structure of the dual space ${\mathcal
Y}^*$ of a normed bimodule ${\mathcal Y}$ over a Banach algebra
${\mathcal B}$ is defined via
$$(b\cdot f)(y)=f(yb)\,,\quad (f\cdot b)(y)=f(by) \qquad (b\in{\mathcal
B}, y\in{\mathcal Y}, f\in{\mathcal Y}^*)$$ If $\sigma=\tau$ we
simply use the terminologies $\sigma$-derivation,
$\sigma$-amenability, etc.

We establish some general algebraic properties of
$(\sigma,\tau)$-derivations and investigate the relation between
amenability and $(\sigma,\tau)$-amenability of Banach algebras. In
particular, we prove that if $\mathfrak A$ is a $C^*$-algebra and
$\sigma:{\mathfrak A}\to {\mathfrak A}$ is a $*$-homomorphism, then
${\mathfrak A}$ is $\sigma$-amenable if and only if
$\sigma({\mathfrak A})$ is amenable.

For definitions and elementary properties of Banach algebras we
refer the reader to \cite{DAL, PAL}.


\section{Algebraic Aspects of  $(\sigma,\tau)$-derivations}

Throughout this section let ${\mathcal A}$ be a unital algebra with
unit $e$, let $\sigma,\tau:{\mathcal A}\to{\mathcal A}$ be linear
mappings such that both $\sigma({\mathcal A})$ and $\tau({\mathcal
A})$ be subalgebras of ${\mathcal A}$ and ${\mathcal X}$ is a
$\big(\tau({\mathcal A}),\sigma({\mathcal A})\big)$-bimodule. In the
case where $e\in \sigma({\mathcal A})\cup \tau({\mathcal A})$, we
assume that ${\mathcal X}$ is unit linked, i.e.
$ex=xe=e\,\,(x\in{\mathcal X})$. In this section we first give a
sufficient condition under which each $(\sigma,\tau)$-derivation on
$\mathcal A$ into $\mathcal X$ is of the form $D_x^+$ for some
$x\in{\mathcal X}$. We second show how the assumption that $\sigma$
and $\tau$ are homomorphisms ensures $D_x^+=0$ for each $x\in
{\mathcal X}$, and provide a suitable setting for study of
$(\sigma,\tau)$-derivations as well. Let us start our work with the
following definition.


\begin{definition} Let
$x$ be a fixed element of ${\mathcal X}$. A linear mapping
$\varphi:{\mathcal A}\to{\mathcal A}$ is called a right (left,
resp.) $x$-homomorphism if
$x\cdot(\varphi(ab)-\varphi(a)\varphi(b))=0$ (
$(\varphi(ab)-\varphi(a)\varphi(b))\cdot x=0$, resp.) for each
$a,b\in\mathcal A$.
\end{definition}


\begin{theorem} Suppose $\sigma(e)=\lambda e$ and $\tau(e)=(1-\lambda)e$ for some
nonzero number $\lambda\neq 1$. Then for each
$(\sigma,\tau)$-derivation $D:{\mathcal A}\to{\mathcal X}$ there is
an $x\in{\mathcal X}$ such that
\[D(a)=x\cdot
\frac{\sigma(a)}{\lambda}=\frac{\tau(a)}{1-\lambda}\cdot x \qquad
(a\in\mathcal A).\] In this case, $\frac{\sigma}{\lambda}$ is a
right $x$-homomorphism and $\frac{\tau}{1-\lambda}$ is a left
$x$-homomorphism.
\end{theorem}
\begin{proof} It follows from
\[D(a)=D(a)\cdot\sigma(e)+\tau(a)\cdot D(e)=\lambda
D(a)+\tau(a)\cdot D(e)\] that
\[(1-\lambda)D(a)=\tau(a)\cdot D(e).\]
Now if $D(e)=0$ then $D(a)=0$ for each $a\in\mathcal A$. Hence we
may suppose that $x:=D(e)\neq0$. Thus $D(a)=\frac{\tau(a)\cdot
x}{1-\lambda}$. A similar argument shows that
$D(a)=\frac{x\cdot\sigma(a)}{\lambda}$. Now we have
\begin{eqnarray*} \frac{\tau(ab)}{1-\lambda}\cdot x&=&D(ab)\\
&=&D(a)\cdot\sigma(b)+\tau(a)\cdot D(b)\\
&=&\frac{\tau(a)\cdot
x}{1-\lambda}\cdot\sigma(b)+\tau(a)\cdot\frac{\tau(b)\cdot
x}{1-\lambda}\\
&=&\frac{\lambda\tau(a)
}{1-\lambda}\cdot\frac{x\cdot\sigma(b)}{\lambda}+\frac{\tau(a)\tau(b)\cdot
x}{1-\lambda}\\
&=&\frac{\lambda\tau(a) }{1-\lambda}\cdot\frac{\tau(b)\cdot
x}{1-\lambda}+\frac{\tau(a)\tau(b)\cdot
x}{1-\lambda}\\
&=&\frac{\tau(a)\tau(b)}{(1-\lambda)^2}\cdot
x\\
&=&\frac{\tau(a)}{1-\lambda}\frac{\tau(b)}{1-\lambda}\cdot x.
\end{eqnarray*}
We can therefore deduce that $\frac{\tau}{1-\lambda}$ is a left
$x$-homomorphism. By a similar argument we see that
$\frac{\sigma}{\lambda}$ is a right $x$-homomorphism.
\end{proof}


\begin{corollary} Suppose that $\sigma(e)=\lambda e$ and $\tau(e)=(1-\lambda)e$ for some
nonzero number $\lambda\neq 1$. Then for each $\big(\tau({\mathcal
A}),\sigma({\mathcal A})\big)$-bimodule $\mathcal X$ and each
$(\sigma,\tau)$-derivation $D:{\mathcal A}\to{\mathcal X}$, there
are an $x\in{\mathcal X}$ and two linear mappings
$\Sigma,T:{\mathcal A}\to{\mathcal A}$ such that $2\Sigma$ is a
right $x$-homomorphism, $2T$ is a left $x$-homomorphism, $D$ is a
$(\Sigma,T)$-derivation and
\[D(a)=2x\cdot\Sigma(a)=2T(a)\cdot x=
x\cdot\Sigma(a)+T(a)\cdot x\qquad (a\in\mathcal A).\]
\end{corollary}
\begin{proof} Put $\Sigma=\frac{\sigma}{2\lambda}$ and
$T=\frac{\tau}{2(1-\lambda)}$.
\end{proof}
Thus we have


\begin{theorem} Let $\mathcal A$ be a unital algebra with unit
$e$. Suppose that $\sigma,\tau:{\mathcal A}\to{\mathcal A}$ are
linear mappings such that both $\sigma({\mathcal A})$ and
$\tau({\mathcal A})$ are subalgebras of ${\mathcal A}$ and that
$\sigma(e)=\lambda e$ and $\tau(e)=(1-\lambda)e$ for some nonzero
number $\lambda\neq 1$. Then every $(\sigma,\tau)$-derivation is
$(\sigma,\tau)$-inner.
\end{theorem}


\begin{proposition}\label{Madjid} Let $x$ be a fixed element
of $\mathcal X$. Then $D_x^-(a)=x\cdot\sigma(a)-\tau(a)\cdot x$ is a
$(\sigma,\tau)$-derivation if and only if
\[x\cdot(\sigma(ab)-\sigma(a)\sigma(b))=(\tau(ab)-\tau(a)\tau(b))\cdot
x\] for all $a,b\in{\mathcal A}$.
\end{proposition}
\begin{proof}
$D_x^-$ is a $(\sigma,\tau)$-derivation if and only if
\begin{eqnarray*}
x\sigma(ab)-\tau(ab)x&=&D_x^-(ab)\\&=&D_x^-(a)\sigma(b)+\tau(a)D_x^-(b)\\
&=&\big(x\sigma(a)-\tau(a)x\big)\sigma(b)+\tau(a)\big(x\sigma(b)-\tau(b)x\big)\,,
\end{eqnarray*}
or equivalently
\[x\cdot(\sigma(ab)-\sigma(a)\sigma(b))=(\tau(ab)-\tau(a)\tau(b))\cdot
x\]
\end{proof}


Let $\mathcal X\sigma(\mathcal A):=\{x\sigma(a): x\in{\mathcal X},
a\in{\mathcal A}\}=\{0\}$ (resp. $\tau(\mathcal A)\mathcal X=\{0\}$)
and $a\mathcal X=\{0\}$ (resp. $\mathcal X a=\{0\}$) implies that
$a=0$. If all mappings $D_x^-$ are $(\sigma,\tau)$-derivation, then
Proposition \ref{Madjid} implies that $\sigma$ (resp. $\tau$) is
necessarily a homomorphism. This pathological observation and the
following proposition let us reasonably assume that $\sigma$ and
$\tau$ are unital homomorphisms when we deal with
$(\sigma,\tau)$-amenability (or $(\sigma,\tau)$-contractibility).


\begin{proposition}\label{dplus} Let $\sigma,\tau:{\mathcal
A}\to{\mathcal A}$ be two unital homomorphisms. Then
$D_x^+:{\mathcal A}\to{\mathcal X}$ defined by
$D_x^+(a)=x\cdot\sigma(a)+\tau(a)\cdot x$ is a
$(\sigma,\tau)$-derivation if and only if $D_x^+=0$.
\end{proposition}
\begin{proof} We have
\begin{eqnarray*} &&x\cdot\sigma(a)\sigma(b)+\tau(a)\tau(b)\cdot x\\&=&
x\cdot\sigma(ab)+\tau(ab)\cdot x\\&=&D_x^+(ab)\\
&=&D_x^+(a)\cdot\sigma(b)+\tau(a)\cdot D_x^+(b)\\
&=&(x\cdot\sigma(a)+\tau(a)\cdot
x)\cdot\sigma(b)+\tau(a)\cdot(x\cdot\sigma(b)+\tau(b)\cdot
x)\\
&=&x\cdot\sigma(a)\sigma(b)+\tau(a)\cdot
x\cdot\sigma(b)+\tau(a)\cdot x\cdot\sigma(b)+\tau(a)\tau(b)\cdot
x.\end{eqnarray*} Thus $\tau(a)\cdot x\cdot\sigma(b)=0$ for all
$a,b\in\mathcal A$. putting $a=e$ we have $x\cdot\sigma(b)=0$ and
putting $b=e$ we have $\tau(a)\cdot x=0$. Thus
$D_x^+(a)=x\cdot\sigma(a)+\tau(a)\cdot x=0$.
\end{proof}


\section{$\sigma$-amenability of $C^*$-algebras}

Throughout this section, let ${\mathcal A}$ be a Banach algebra, let
$\sigma,\tau:{\mathcal A}\to{\mathcal A}$ be homomorphisms and let
${\mathcal X}$ be a Banach $\big(\tau({\mathcal A}),\sigma({\mathcal
A})\big)$-bimodule. We also assume that all mappings under
consideration are linear and continuous. In this section we study
some interrelations between $(\sigma,\tau)$-amenability and ordinary
amenability. Let us start our works with the following easy
observation.

\begin{proposition}\label{cor1}
If $\mathcal A$ is amenable (contractible), then  $\mathcal A$ is
$(\sigma,\tau)$-amenable ($(\sigma,\tau)$-contractible) for every
two homomorphisms $\sigma$ and $\tau$.
\end{proposition}
\begin{proof}
We only prove the statement concerning amenability. The other case
can be similarly proved.\\
Let ${\mathcal X}$ be a Banach $\big(\tau({\mathcal
A}),\sigma({\mathcal A})\big)$-bimodule and $D:{\mathcal
A}\longrightarrow{\mathcal X^*}$ be a $(\sigma,\tau)$-derivation.
One can regard ${\mathcal X}$ as a Banach ${\mathcal A}$-bimodule
via
\begin{eqnarray*}
a\cdot x=\tau(a)x\,,\quad  x\cdot a=x\sigma(a)\qquad (a\in{\mathcal
A}, x\in{\mathcal X}).
\end{eqnarray*}
Then $D(ab)=D(a)\sigma(b)+ \tau(a)D(b)=D(a)\cdot b+\tau(a)\cdot
D(b)$. Therefore $D$ is a derivation, and so, by amenability of
${\mathcal A}$, there exists $f\in{\mathcal X}^*$ such that
$D(a)=f\cdot a-a\cdot f=f\sigma(a)-\tau(a)f$.
\end{proof}


An interesting question is whether we can have a similar result but
in contrary to Proposition \ref{cor1}.

\begin{example} Let ${\mathcal A}$ be a non amenable Banach algebra, let $\tilde{{\mathcal
A}}={\mathbb C}\oplus{\mathcal A}$ be its unitization and
$\sigma:\tilde{{\mathcal A}}\to\tilde{{\mathcal A}}$ be defined by
$\sigma(z\oplus a)=z$ for all $z\in{\mathbb C}$ and $a\in{\mathcal
A}$. Then $\tilde{{\mathcal A}}$ is not amenable (see \cite{RUN})
but is $\sigma$-amenable. To prove the $\sigma$-amenability of
$\tilde{{\mathcal A}}$ let ${\mathcal X}$ be a Banach space and
$D:\tilde{{\mathcal A}}\to{\mathcal X}$ be a $\sigma$-derivation.
Define $d:{\mathbb C}\to{\mathcal X}$ by $d(z)=D(z\oplus0)$. Then
$d$ is an ordinary derivation on ${\mathbb C}$, which is $0$. Thus
$D(z\oplus a)=D(z\oplus 0)+D(0\oplus a)=d(z)+D(0\oplus a)=D(0\oplus
a)$. Now we have
\begin{eqnarray*}
D(zw\oplus zb+wa+ab)&=&D((z\oplus a)(w\oplus b))\\&=&D(z\oplus
a)\sigma(w\oplus b)+\sigma(z\oplus a)D(w\oplus b)\\&=&D(0\oplus
a)w+zD(0\oplus b).
\end{eqnarray*} Putting $z=w=0$ and $b=1$ we get $D(0\oplus
a)=0$. Thus $D(z\oplus a)=D(0\oplus a)=0$ for all $z\in{\mathbb C}$
and $a\in{\mathcal A}$.
\end{example}


\section{Correspondence Between $(\sigma,\tau)$-Derivations and
Derivations on $C^*$-Algebras}

In this section we fix a $C^*$-algebra ${\mathfrak A}$ and a Banach
$\big(\tau({\mathfrak A}),\sigma({\mathfrak A})\big)$-bimodule
$\mathcal X$ and give a correspondence between
$(\sigma,\tau)$-derivations on ${\mathfrak A}$ and ordinary
derivations on a certain algebra related to ${\mathfrak A}$. To
state the result, we give some terminology.

Let $\sigma,\tau:{\mathfrak A}\to{\mathfrak A}$ be two
$*$-homomorphisms. The mapping $\psi: {\mathfrak A} \to {\mathfrak
A} \oplus {\mathfrak A}$ defined by
$\psi(a)=\big(\sigma(a),\tau(a)\big)$ is a $*$-homomorphism, where
$\oplus$ denotes the $C^*$-direct sum. Hence ${\mathfrak
A}_{(\sigma,\tau)}:=\psi({\mathfrak A})=\{(\sigma(a),\tau(a)):
a\in{\mathfrak A}\}$ is a $C^*$-algebra which is isometrically
isomorphic to ${\mathfrak A}/{\ker(\sigma)\cap\ker(\tau)}$. If
$\mathcal X$ is a Banach $\big(\tau({\mathfrak A}),\sigma({\mathfrak
A})\big)$-bimodule, then $\mathcal X$ can be regarded as an
${\mathfrak A}_{(\sigma,\tau)}$-bimodule via the operations
\begin{eqnarray*}
x\cdot(\sigma(a),\tau(a))=x\cdot\sigma(a)\,,\quad
(\sigma(a),\tau(a))\cdot x&=&\tau(a)\cdot x\qquad (a\in {\mathfrak
A}, x\in {\mathcal X}).
\end{eqnarray*}
Then to each linear mapping $D:{\mathfrak A}\to\mathcal X$ there
corresponds a linear mapping $d:{\mathfrak
A}_{(\sigma,\tau)}\to\mathcal X$ by $d(\sigma(a),\tau(a))=D(a)$ in
such a way that $D$ is a $(\sigma,\tau)$-derivation if and only if
$d$ is an ordinary derivation. Now we are in the situation to
present the results of this section.

\begin{proposition} Each $(\sigma,\tau)$-derivation $D:{\mathfrak
A}\to{\mathcal X}$ is $(\sigma,\tau)$-inner if and only each
derivation $d:{\mathfrak A}_{(\sigma,\tau)}\to{\mathcal X}$ is
inner.
\end{proposition}

\begin{proof} Let each $(\sigma,\tau)$-derivation
from ${\mathfrak A}$ to ${\mathcal X}$ be $(\sigma,\tau)$-inner.
Suppose that $d:{\mathfrak A}_{(\sigma,\tau)}\to{\mathcal X}$ is a
derivation. Then as mentioned above $D(a)=d(\sigma(a),\tau(a))$ is a
derivation from ${\mathfrak A}$ to ${\mathcal X}$. Hence there
exists an $x\in {\mathcal X}$ such that
$D(a)=x\cdot\sigma(a)-\tau(a)\cdot x$ for all $a\in {\mathfrak A}$.
This implies that
\[d(\sigma(a),\tau(a))=D(a)=x\cdot(\sigma(a),\tau(a))-(\sigma(a),\tau(a))\cdot
x\] for all $(\sigma(a),\tau(a))\in{\mathfrak A}_{(\sigma,\tau)}$.
Therefore $d$ is inner.

Conversely, let each derivation from ${\mathfrak A}_{(\sigma,\tau)}$
to ${\mathcal X}$ be inner. Suppose that $D:{\mathfrak
A}\to{\mathcal X}$ is a $(\sigma,\tau)$-derivation. Then again as
above $d(\sigma(a),\tau(a))=D(a)$ is a derivation from ${\mathfrak
A}_{(\sigma,\tau)}$ to ${\mathcal X}$. So there is an $x\in{\mathcal
X}$ such that $d(u)=x\cdot u-u\cdot x$ for all $u\in{\mathfrak
A}_{(\sigma,\tau)}$, or equivalently
$D(a)=d(\sigma(a),\tau(a))=x\cdot\sigma(a)-\tau(a)\cdot x$. Thus $D$
is $(\sigma,\tau)$-inner.
\end{proof}


\begin{corollary} Let ${\mathfrak A}$ be a $C^*$-algebra, $\sigma:{\mathfrak
A}\to {\mathfrak A}$ be a $*$-homomorphism and $\mathcal X$ be a
Banach $\sigma({\mathcal A})$-bimodule. Then each
$\sigma$-derivation $D:{\mathfrak A}\to{\mathcal X}$ is
$\sigma$-inner if and only each derivation $d:\sigma({\mathfrak
A})\to{\mathcal X}$ is inner.
\end{corollary}
\begin{proof} Use the fact that ${\mathfrak A}_{(\sigma,\sigma)}={\mathfrak
A}/\ker(\sigma)=\sigma({\mathfrak A})$.
\end{proof}


Under some restrictions on $\sigma$ and $\tau$ we mention our last
result concerning $(\sigma,\tau)$-amenability of ${\mathfrak A}$.

\begin{theorem}
Let ${\mathfrak A}$ be a $C^*$-algebra and let
$\sigma,\tau:{\mathfrak A}\to {\mathfrak A}$ be two
$*$-homomorphisms with $\ker(\sigma)=\ker(\tau)$. Then ${\mathfrak
A}$ is $(\sigma,\tau)$-amenable ($(\sigma,\tau)$-contractible,
resp.) if and only if $\sigma({\mathfrak A})$ is amenable
(contractible, resp.).
\end{theorem}
\begin{proof} Since $\ker(\sigma)=\ker(\tau)$ the relations
\begin{eqnarray*} x\cdot\sigma(a)&=&x\cdot(\sigma(a),\tau(a)),\\
\tau(a)\cdot x&=&(\sigma(a),\tau(a))\cdot x\end{eqnarray*} ensure us
to say that $\mathcal X$ is a Banach $\big(\tau({\mathfrak
A}),\sigma({\mathfrak A})\big)$-bimodule if and only if it is a
${\mathfrak A}_{(\sigma,\tau)}$-bimodule. Thus the correspondence
between $(\sigma,\tau)$-derivations on $\mathfrak A$ and derivations
on ${\mathfrak A}_{(\sigma,\tau)}$ helps us to complete the proof.
Note that in this case ${\mathfrak A}_{(\sigma,\tau)}$ is
\[{\mathfrak A}/{\ker(\sigma)\cap\ker(\tau)}={\mathfrak
A}/\ker(\sigma)=\sigma({\mathfrak A}).\]
\end{proof}

The following corollary includes a converse to Proposition
\ref{cor1} in the framework of $C^*$-algebras. The general case
where ${\mathcal A}$ is an arbitrary Banach algebra remains open.

\begin{corollary}
Let ${\mathfrak A}$ be a $C^*$-algebra and let $\sigma:{\mathfrak
A}\to {\mathfrak A}$ be a $*$-homomorphism. Then ${\mathfrak A}$ is
$\sigma$-amenable ($\sigma$-contractible, resp.) if and only if
$\sigma({\mathfrak A})$ is amenable (contractible, resp.).
\end{corollary}

\begin{remark} If $\sigma$ is not homomorphism, then
the above result is not true in general. To see this, let $\mathfrak
A$ be an amenable $C^*$-algebra and $\sigma=\frac12 id_{\mathfrak
A}$, where $id_{\mathfrak A}$ is the identity mapping on $\mathfrak
A$. Note that $\sigma({\mathfrak A})={\mathfrak A}$. Let $\mathcal
X$ be a symmetric ${\mathfrak A}$-bimodule and define $D:{\mathfrak
A}\to{\mathcal X}^*$ by $D(a)=f\cdot a$, where $f$ is a fixed
element of ${\mathcal X}^*$. Then $D$ is a $\sigma$-derivation,
since
\[D(ab)=f\cdot ab=(f\cdot a)\cdot\frac{b}2+\frac{a}2\cdot(f\cdot
b)=D(a)\cdot\sigma(b)+\sigma(a)\cdot D(b).\] But $D$ is not inner.
In contrary, suppose that $D(a)=g\cdot\sigma(a)-\sigma(a)\cdot g$
for some fixed element $g$ of ${\mathcal X}^*$. Since $\mathcal X$
is symmetric, we have $D=0$ which contradicts to the definition of
$D$. Thus ${\mathfrak A}$ is not $\sigma$-amenable.
\end{remark}

\textbf{Acknowledgement.}  The author would like to thank the
anonymous referee for some useful comments improving the paper.

\end{document}